\documentclass[a4paper,12pt]{amsart}  
\usepackage{a4wide,amsmath,amssymb,amsthm}     
\usepackage[utf8]{inputenc}
\usepackage{enumitem}
\usepackage{color}
\theoremstyle{plain}      
 
\newtheorem{thm}{Theorem}[section]

\newtheorem{lemma}[thm]{Lemma}     
\newtheorem{prop}[thm]{Proposition}

\theoremstyle{definition}      
     
\newtheorem{definition}[thm]{Definition}  
 
\newtheorem{remark}[thm]{Remark}

\DeclareMathAlphabet{\doba}{U}{msb}{m}{n}

\gdef\mR{\doba{R}}

\def\di{{\rm d}}

\def\vol{{\mathrm{vol}}}

\def\Sim{\mathrm{Sim}}

\let\<\langle 
\let\>\rangle

\newcommand{\definedas}{\mathrel{\raise.095ex\hbox{\rm :}\mkern-5.2mu=}}

\subjclass[2010]{53C18}

\keywords{Conformal geometry, Weyl structure, Gauduchon metric, LCP structure}

\begin{document}   
	
	\title{Adapted metrics on locally conformally product manifolds}

	\begin{abstract}
		We show that the Gauduchon metric $g_0$ of a compact locally conformally product manifold $(M,c,D)$ of dimension greater than $2$ is adapted, in the sense that the Lee form of $D$ with respect to $g_0$ vanishes on the $D$-flat distribution of $M$. We also characterize adapted metrics as critical points of a natural functional defined on the conformal class.
	\end{abstract}
	
	\author{Andrei Moroianu, Mihaela Pilca}

\address{Andrei Moroianu \\ Université Paris-Saclay, CNRS,  Laboratoire de mathématiques d'Orsay, 91405, Orsay, France, and Institute of Mathematics “Simion Stoilow” of the Romanian Academy, 21
Calea Grivitei, 010702 Bucharest, Romania}
\email{andrei.moroianu@math.cnrs.fr}

\address{Mihaela Pilca\\Fakult\"at f\"ur Mathematik\\
Universit\"at Regensburg\\Universit\"atsstr. 31 
D-93040 Regensburg, Germany}
\email{mihaela.pilca@mathematik.uni-regensburg.de}

	\maketitle

	\section{Introduction}
	
	A similarity structure on a compact manifold $M$ is a conformal class $c$ of Riemannian metrics, and a closed, non-exact Weyl connection $D$. Equivalently, a similarity structure can be defined by a Riemannian metric $h$ on the universal cover $\widetilde M$ such that the fundamental group $\pi_1(M)$ consists of homotheties of $h$, not all of them being isometries. 
	
	The correspondence between the two definitions goes roughly as follows. If $D$ is a closed, non-exact, Weyl connection on $(M,c)$, its lift $\widetilde D$ to the universal cover $(\widetilde M,\widetilde c)$ is exact, so it is the Levi-Civita connection of some Riemannian metric $h\in \widetilde c$. Each element of $\pi_1(M)$ is affine ({\em i.e.} preserves $\widetilde D=\nabla^h$), so is a homothety of $h$. Moreover, not all elements in $\pi_1(M)$ are isometries of $h$, since otherwise $D$ would be exact. The converse is similar.
	
	A locally conformally product (LCP) structure on a compact manifold $M$ is a similarity structure $(c,D)$ such that $D$ has reducible holonomy representation (or, equivalently, such that the corresponding metric $h$ on $\widetilde M$ is reducible). In general we will tacitly assume that the metric $h$ of an LCP structure is non-flat. The flat case is of completely different nature, and was classified in \cite{fried}.
		
	The theory of LCP manifolds was founded by the seminal work of M. Kourganoff \cite{k}, where it is proved, among many other results, that if $(\widetilde M,h)$ is non-flat, then it has a global de Rham decomposition $(\widetilde M,h)=\mR^q\times (N,g_N)$, where $\mR^q$ is a flat Euclidean space with $q\ge 1$ and $(N,g_N)$ is an irreducible, incomplete Riemannian manifold. This result is highly non-trivial, and does not follow from the classical de Rham theorem, since $(\widetilde M, h)$ is incomplete. Moreover, it is noteworthy that, unlike the classical case of reducible complete simply connected Riemannian manifolds, where the number of factors is arbitrary, the universal cover of LCP manifolds has always exactly two factors.
	
	Since the fundamental group preserves the above decomposition of $\widetilde M$, the distributions $\mathrm{T}\mR^q$ and $\mathrm{T}N$ define orthogonal $D$-parallel distributions $T_0$ and $T_1$ on $M$. We will call $T_0$ the flat distribution and $T_1$ the non-flat distribution of $M$. 
	
	For every metric $g\in c$ on $M$, the Lee form of $D$ with respect to $g$ will be denoted by $\theta_g$ (or simply $\theta$ when there is no risk of confusion). By definition, $\theta_g$ is closed and non-exact. The metric $g$ is called Gauduchon if $\theta_g$ is co-closed, \emph{i.e.} $\delta^g\theta_g=0$, where $\delta^g$ denotes the codifferential with respect to $g$. By \cite{g}, Gauduchon metrics always exist, and are unique up to constant rescaling. 
	
	On LCP manifolds $(M,c,D)$, there is another class of distinguished metrics, introduced by B. Flamencourt \cite{brice}. A metric $g\in c$ is called {\em adapted} if $\theta_g$ vanishes on the flat distribution $T_0$. Using a convolution and smoothing argument, Flamencourt showed that every LCP manifold carries adapted metrics \cite[Prop. 3.6]{brice}. 
	
	Our first main result is to show that in fact the Gauduchon metric of every LCP structure $(M,c,D)$ is adapted. In particular, this implies that every LCP manifold carries adapted metrics. The proof is based on an averaging argument, using several deep results by Kourganoff concerning the structure of the group of similarities of $\widetilde M$. 
	
	In the last section we show that every critical point of the functional  $\mathcal F$ which to any metric $g$ in the conformal class $c$ associates the integral of the restriction to $T_0$ of the Lee form $\theta_g$ is an adapted metric, and thus a global minimum of $\mathcal F$. This result was inspired by some similar statements for LCK metrics obtained in \cite{aiot}.
	
	{\bf Acknowledgments.} This work was supported by the Procope Project No. 57650868 (Germany) /  48959TL (France) and by the PNRR Project CF149/31.07.2023.

	\section{Preliminaries}
	
	\subsection{Weyl connections}
		A {\it Weyl structure} on a conformal manifold $(M,c)$ is a torsion-free linear connection $D$ which preserves the conformal class $c$. 
	The condition that $D$ preserves the conformal class $c$ means that for each metric $g\in c$, there exists a unique $1$-form $\theta_g\in\Omega^1(M)$, called the {\it Lee form} of $D$ with respect to $g$, such that 
	\begin{equation}\label{dg}Dg=-2\theta_g\otimes g.
	\end{equation}
	The Weyl connection $D$ is then related to the Levi-Civita covariant derivative $\nabla^g$ by
	\begin{equation}\label{weylstr}
		D_X=\nabla^g_X+\theta_g(X)\mathrm{Id} + \theta_g\wedge X, \quad \forall X\in  \mathrm{T}M,
	\end{equation}
	where $\theta_g\wedge X$ is the skew-symmetric endomorphism of $\mathrm{T}M$ defined by $$(\theta_g\wedge X)(Y):=\theta_g(Y) X-g(X,Y)(\theta_g)^\sharp.$$
	
	A Weyl connection $D$ is called {\it closed} if it is locally the Levi-Civita connection of a (local) metric in $c$ and is called {\it exact} if it is the Levi-Civita connection of a globally defined metric in~$c$. Equivalently, $D$ is closed (resp. exact) if its Lee form with respect to one (and hence to any) metric in $c$ is  closed (resp. exact).	If the manifold $M$ is compact of dimension greater than $2$, then for every Weyl connection~$D$ on $(M,c)$ there exists a unique (up to constant rescaling) metric $g_0\in c$, called the {\it Gauduchon metric} of $D$, such that its associated Lee form $\theta_0$ is coclosed with respect to $g_0$, cf.~\cite{g}.
	
	Let us recall the formula for the conformal change of the codifferential and the Laplacian, which will be needed in the sequel. If two metrics on $M$ are conformally related, $h=e^{2\varphi}g$, then their codifferentials and Laplacians are related as follows (\cite[Thm. 1.159]{besse}):
	\begin{equation}\label{codiffconf}
		\delta^h =e^{-2\varphi}\delta^g + (2-n)e^{-2\varphi} \iota_{\theta^{\#_g}}.
	\end{equation}
	\begin{equation}\label{laplacianconf}
	\Delta^h f=e^{-2\varphi}(\Delta^g f+(2-n)g(\di \varphi,\di f)), \quad \forall f\in\mathcal{C}^\infty(M).
	\end{equation}

	\subsection{LCP structures}
	
	Let $(M,c)$ be a compact conformal manifold. As explained in the introduction, an LCP structure on $(M,c)$ is a closed, non-exact Weyl connection $D$ with reducible holonomy. If $(M, c, D)$ is a compact LCP manifold, its universal cover $(\widetilde M, h)$ is endowed with a Riemannian metric $h$, determined up to a multiplicative factor by the fact that its Levi-Civita connection is the pull-back of $D$ to $\widetilde M$. 
	
	By \cite[Thm. 1.5]{k}, $(\widetilde M, h)$ is globally isometric to a Riemannian product $\mR^q\times (N,g_N)$, where $\mR^q$ is the flat Euclidean space and $(N,g_N)$ is a non-flat, incomplete and irreducible Riemannian manifold.  The fundamental group of $M$ acts on $\widetilde M$ by homotheties, and we denote by $\rho$ the corresponding representation:  $$\rho\colon \pi_1(M)\to\mR_+^*, \quad \alpha^*h= \rho(\alpha)h, \quad \text{ for all } \alpha\in\pi_1(M),$$
	 so that $\rho(\alpha)^{1/2}$ is the similarity ratio of $\alpha\in \pi_1(M)$. 
	 
	 For any real number $t$, one may define the representation $\rho^t\colon \pi_1(M)\to\mR_+^*$ by setting  $\rho^t(\alpha):=(\rho(\alpha))^t$, for all $\alpha\in \pi_1(M)$. 
	A function $f\in\mathcal{C}^\infty(\widetilde M)$ is called $\rho^t$-equivariant, if  $\alpha^*f= \rho^t(\alpha)f$, for all $\alpha\in \pi_1(M)$. Let us denote by $\mathcal{C}^\infty_{\rho^t}(\widetilde M)$ the set of all $\rho^t$-equivariant functions.
	
	With respect to the decomposition $\widetilde M=\mR^q\times N$, every element of $\pi_1(M)$ can be written as $\alpha=(\alpha_0,\alpha_1)$, with $\alpha_0\in\Sim(\mR^q)$ and $\alpha_1\in\Sim(N,g_N)$, where $\Sim(N,g_N)$ denotes the group of similarities of $(N,g_N)$, namely:
	$$\Sim(N,g_N):=\{\phi\colon N\to N \, |\, \phi \text{ is a diffeomorphism} \text{ and } \exists \lambda>0 \text{ s.t. } \phi^* g_N=\lambda^2 g_N \}.$$
	 Following \cite[\S4]{k}, we introduce the group obtained as the projection of $\pi_1(M)$ onto the subgroup of homotheties of $(N,g_N)$:
	$$P:=\{\alpha_1\in\Sim(N,g_N)\ |\ \exists\alpha_0\in \Sim(\mR^q),\ (\alpha_0,\alpha_1)\in\pi_1(M)\}.
	$$
	We denote by $\overline P$ its closure in $\Sim(N,g_N)$ and by $\overline P^0$ the connected component of the identity in $\overline P$. By \cite[Lemmas 4.1 and  4.13]{k}, $\overline P_0$ is an abelian subgroup of $\mathrm{Isom}(N,g_N)$.
	 
	 The subgroup  $\Gamma_0:=\pi_1(M)\cap (\mathrm{Sim}(\mR^q)\times \overline{P}^{0})$ of  $\pi_1(M)$ is a full lattice in $\widetilde G:=\mR^q\times \overline{P}^{0}$, by \cite[Lemma 4.18]{k}. Moreover, $\Gamma_0\subset \mathrm{Isom}(\widetilde M, h)$ since $\overline P_0\subset\mathrm{Isom}(N,g_N)$. Let $G$ denote the compact quotient $G:=\widetilde{G}/\Gamma_0$.

	\begin{remark}
		Let us observe that for any $t\in\mR$, every function in $\mathcal{C}^\infty_{\rho^t}(\widetilde M)$  is $\Gamma_0$-invariant. This follows directly from the fact that $\Gamma_0$ acts by isometries with respect to the metric $h$, \emph{i.e.} $\rho(\gamma_0)=1$, for all $\gamma_0\in \Gamma_0$. Hence, we have for any $f\in \mathcal{C}^\infty_{\rho^t}(\widetilde M)$ and $\gamma_0\in \Gamma_0$: $\gamma_0^* f =(\rho(\gamma_0))^t f=f$. Thus, the compact group $G$  naturally  acts on the space of  $\rho^t$-equivariant functions on $\widetilde M$:
		\[\overline\gamma^*f:=\gamma^*f, \quad \text{ for all } \overline\gamma\in G \text{ and }  f\in  \mathcal{C}^\infty_{\rho^t}(\widetilde M),\]
		where $\gamma$ projects onto $\overline\gamma$ via the canonical surjection from $\widetilde G$ to $G=\widetilde G/\Gamma_0$.
		\end{remark}
		
		\section{An averaging argument}
	
	Let $\mathrm{vol}^G$ denote the left  invariant  Haar measure of the compact group $G$ defined in the previous section (or equivalently the bi-invariant Haar measure, because $G$ is abelian), which is uniquely defined up to a scalar multiple.
	The above remark allows us to associate to every $\rho^t$-equivariant function $f$ on $\widetilde M$ another function  on $\widetilde M$, denoted by $f^G$, as follows:
	\begin{equation}\label{deff^G}
	 f^G (x):=\int_G (\overline\gamma^* f)(x)\,\mathrm{vol}^G,\qquad \text{ for all } x\in \widetilde M.
	 \end{equation}
	
	For $\alpha\in\pi_1(M)$, let $\varphi_\alpha$ denote the conjugation with $\alpha$ acting on $G$:
	$$\varphi_\alpha\colon G\to G, \quad \varphi_\alpha(\overline\gamma):=\alpha^{-1}\overline{\gamma}\alpha.$$
	For $\overline\gamma\in G$, let $L_{\overline\gamma}$ denote the left multiplication with $\overline\gamma$. Then the following equality holds:
	\begin{equation}\label{leftconj}
	L_{\varphi_\alpha(\overline\gamma)}=\varphi_\alpha L_{\overline\gamma}\varphi_\alpha^{-1}, \quad\forall \overline\gamma\in G, \alpha\in\pi_1(M).
	\end{equation}
Indeed, for all $\overline\gamma'\in G$ we compute: 
\begin{equation*}
\begin{split}\varphi_\alpha L_{\overline\gamma}\varphi_\alpha^{-1}(\overline\gamma')&=\varphi_\alpha L_{\overline\gamma}(\alpha\overline{\gamma}'\alpha^{-1})=\varphi_\alpha(\overline\gamma\alpha\overline{\gamma}'\alpha^{-1})=\alpha^{-1}(\overline\gamma\alpha\overline{\gamma}'\alpha^{-1})\alpha\\
&=(\alpha^{-1}\overline\gamma\alpha)\overline{\gamma}'=L_{\varphi_\alpha(\overline\gamma)}(\overline{\gamma}').
\end{split}
\end{equation*}

	\begin{remark}
		 For all $\alpha\in\pi_1(M)$, the following equality holds:
		\begin{equation}\label{haarinv}
		\varphi_\alpha^*\mathrm{vol}^G =\mathrm{vol}^G.
		\end{equation}
	This can be checked as follows. We first compute:
	\begin{equation*}
		\begin{split}
			(\varphi^{-1}_\alpha)^*L_{\overline\gamma}^*\varphi_\alpha^*\mathrm{vol}^G=(\varphi_\alpha L_{\overline\gamma} \varphi^{-1}_\alpha)^*\mathrm{vol}^G\overset{\eqref{leftconj}}{=}L^*_{\varphi_\alpha(\overline\gamma)}\mathrm{vol}^G=\mathrm{vol}^G,
		\end{split}	
		\end{equation*}	
where the last equality follows from the fact that the Haar measure $\mathrm{vol}^G$ is by definition left invariant and $\varphi_\alpha(\overline\gamma)\in G$. Hence, for all $\overline\gamma\in G$, we have that $L_{\overline\gamma}^*\varphi_\alpha^*\mathrm{vol}^G=\varphi_\alpha^*\mathrm{vol}^G$, \emph{i.e.} $\varphi_\alpha^*\mathrm{vol}^G$ is also a left invariant Haar measure. Thus, there exists a constant $c_\alpha\in\mR\setminus\{0\}$, such that $\varphi_\alpha^*\mathrm{vol}^G=c_\alpha \mathrm{vol}^G$. Since $\varphi_\alpha$ is a diffeomorphism of the compact group $G$, and $ \mathrm{vol}^G$ is a positive measure, the constant $c_\alpha$ must be equal to $1$.
	\end{remark}	
	
	\begin{lemma}\label{f^G}
		For any function $f\in\mathcal{C}^\infty_{\rho^t}(\widetilde M)$, the function $f^G\in\mathcal{C}^\infty(\widetilde M)$ is ${\rho^t}$-equivariant and $\widetilde G$-invariant. Moreover, if $f$ is a positive function, then $f^G$ is positive too.
	\end{lemma}	

\begin{proof}
	Let $\alpha\in \pi_1(M)$. We compute for all $x\in \widetilde M$:
	\begin{equation*}
		\begin{split}
			(\alpha^* f^G)(x)&=f^G(\alpha x)=\int_G (\overline\gamma^* f)(\alpha x)\,\mathrm{vol}^G=\int_G (\alpha^*\overline\gamma^* f)( x)\,\mathrm{vol}^G\\
			&=\int_G (\alpha^{-1}\overline\gamma\alpha)^*(\alpha^* f)( x)\,\mathrm{vol}^G=\int_G (\alpha^{-1}\overline\gamma\alpha)^*(\rho^t(\alpha) f)( x)\,\mathrm{vol}^G\\
			&\overset{\eqref{haarinv}}{=}\rho^t(\alpha)  \int_G (\varphi_\alpha(\overline\gamma))^*f( x)\,\mathrm{vol}^G=\rho^t(\alpha)  \int_G(\overline{\gamma}^* f)( x)\,\mathrm{vol}^G=\rho^t(\alpha) f^G(x),
		\end{split}	
	\end{equation*}	
which shows that  the function $f^G$ is ${\rho^t}$-equivariant. 
The fact that $f^G$ is $\widetilde G$-invariant follows directly from its definition, since for every $\beta\in\widetilde G$ and $x\in\widetilde M$ we obtain:
\begin{equation*}
		(\beta^* f^G)(x)=\int_G (\overline\beta^*\overline\gamma^* f)( x)\,\mathrm{vol}^G=\int_G (\overline\gamma\overline\beta)^*f\, ( x)\,\mathrm{vol}^G=\int_G (\overline\gamma)^*f\,( x)\,\mathrm{vol}^G= f^G(x).
\end{equation*}	
If $f$ is positive, then $f^G$ is positive, since it is defined as an integral of positive functions.
\end{proof}

\section{The Gauduchon metric of an LCP manifold}	

The purpose of this section is to show that the Gauduchon metric of a compact LCP manifold is adapted.

	\begin{definition} \label{defadapted}
	A metric $g\in c$ on a locally conformally product manifold $(M, c, D)$ is called {\it adapted} if  its Lee form $\theta_g$ vanishes on the flat distribution $T_0$ of $M$, or, equivalently, if its pull-back $\widetilde \theta_g$ vanishes on the flat part $\mR^q$ of the de Rham decomposition of $(\widetilde M, h)$. 
\end{definition}

\begin{remark}\label{remadapted}
	A metric $g\in c$ is adapted if  and only if the function $\varphi$ on $\widetilde M$, which is defined by the equality $h=e^{2\varphi}g$, is a function on $N$.
\end{remark}
	
In the sequel the following characterization of Gauduchon metrics will be needed:

\begin{lemma}\label{charactgauduchon}
Let $(M, c, D)$ be a compact LCP manifold of dimension $n>2$ and let $(\widetilde M, h)$ be its universal cover endowed with the similarity structure $h$. Let $g$ be a metric in the conformal class $c$ and denote by  $\varphi$  the function  on $\widetilde M$ defined by $h=e^{2\varphi}g$. Then the following two statements are equivalent: 
\begin{enumerate}
	\item The metric $g$ is the Gauduchon metric of $(M,c, D)$. 
	\item The function $e^{(2-n)\varphi}$ is harmonic with respect to the metric $h$.
	\end{enumerate}
\end{lemma}

\begin{proof}
	Let $g\in c$. By definition, $g$ is the Gauduchon metric of $(M,c, D)$ if its Lee form $\theta_g$ is $g$-coclosed: $\delta^g\theta_g=0$. Since $\widetilde\theta_g=\di\varphi$, this is equivalent to $\Delta^g \varphi=0$. According to the formula~\eqref{laplacianconf} applied to the function $\varphi$, the equality $\Delta^g \varphi=0$ is equivalent to $ \Delta^h \varphi= (2-n) e^{-2\varphi}|\di\varphi|^2_g$, 
	which can also be written as
	$$\Delta^h \varphi= (2-n) |\di\varphi|^2_h.$$ 
	Furthermore, this equality is equivalent to $\Delta^h (e^{(2-n)\varphi})=0$, as the following computation shows:
	$$ \Delta^h (e^{(2-n)\varphi})=\delta^h \di  (e^{(2-n)\varphi})=(2-n)\delta^h (e^{(2-n)\varphi}\di \varphi)=(2-n)e^{(2-n)\varphi}(\Delta^h \varphi-(2-n)|\di\varphi|^2_h).$$
	Hence, we have proved that $g$ is the Gauduchon metric if and only if  $\Delta^h (e^{(2-n)\varphi})=0$, \emph{i.e.} if and only if $e^{(2-n)\varphi}$ is harmonic with respect to the metric $h$.
\end{proof}	
	
We have now all ingredients necessary for the main result of this section.

	\begin{thm}
		The Gauduchon metric of a compact locally conformally product manifold $(M, c, D)$ of dimension $n>2$ is adapted.
	\end{thm}
	
	\begin{proof}
Let $g_0$ be the Gauduchon metric of $(M, c, D)$ and let $\varphi_0$ be the function  defined on $\widetilde M$ by $h=e^{2\varphi_0}g_0$, where we denote by $g_0$ also the pull-back of the Gauduchon metric to the universal cover. According to Lemma~\ref{charactgauduchon}, we know that 
the positive function $f_0\in \mathcal{C}^\infty(\widetilde M)$ defined by $f_0:=e^{(2-n)\varphi_0}$  is a harmonic function with respect to the metric $h$. Moreover, because $e^{2\varphi_0}$ is $\rho$-equivariant, it follows that $f_0$ is $\rho^{\frac{2-n}{2}}$-equivariant. Applying Lemma~\ref{f^G} to $f_0$ yields that the function $f_0^G$ defined by \eqref{deff^G} is positive, $\rho^{\frac{2-n}{2}}$-equivariant and $\widetilde G$-invariant. Because $\widetilde G$ acts by isometries with respect to the metric $h$, the Laplacian $\Delta^h$ commutes with the pull-back defined by all $\gamma\in\widetilde G$, which implies that $f_0^G$ is still an $h$-harmonic function:
$$ (\Delta^h f_0^G) (x):=\int_G (\Delta^h\overline\gamma^* f_0)(x)\,\mathrm{vol}^G=\int_G (\overline\gamma^*\Delta^h f_0)(x)\,\mathrm{vol}^G=0, \qquad\forall x\in \widetilde M.$$
Since the function $f_0^G$ is positive, we may define the function $\psi:=\frac{1}{2-n}\ln(f_0^G)\in\mathcal{C}^\infty(\widetilde M)$, which thus satisfies 
\begin{equation}\label{psiharmonic}
	\Delta^h (e^{(2-n)\psi})=0.
\end{equation}	
Because $f_0^G$ is $\widetilde G$-invariant, also the function $\psi$ is $\widetilde G$-invariant. Since  the group $\widetilde G=\mR^q\times \overline P^0$ acts by translations on the flat part $\mR^q$ of $\widetilde M=\mR^q\times N$, the fact that $\psi$ is $\widetilde G$-invariant, implies that $\psi$ is a function on $N$.

Since $f_0^G$ is $\rho^{\frac{2-n}{2}}$-equivariant, the function $e^{2\psi}=(f_0^G)^{\frac{2}{2-n}}$ is \mbox{$\rho$-equivariant}. Hence, the metric $g:=e^{-2\psi}h$ is the pull-back of a metric on $M$, which we still denote by $g$ and, by definition, $g\in c$. The fact that $\psi$ is a function on $N$ means that the metric $g$ is adapted, according to Remark~\ref{remadapted}. Applying Lemma~\ref{charactgauduchon}, it follows from \eqref{psiharmonic}, that the metric $g$ is a Gauduchon metric of $(M, c, D)$ and thus, it must coincide, up to a scalar multiple, with $g_0$. This proves that the Gauduchon metric $g_0$ is adapted.
	\end{proof}

\section{Adapted Metrics as Critical Points of a Functional}
In this section we show that adapted metrics on a compact LCP manifold can be characteri\-zed as critical points of a naturally defined functional. 

We first set the notation. Let $(M, c, D)$  be a compact LCP manifold and let $\mathcal{M}_1$ denote the space of metrics in the conformal class $c$ of total volume equal to $1$:
$$\mathcal{M}_1:=\left\{ g\in c \, \Big|\,  \int_M \mathrm{vol}_g=1\right\}.$$

For every $1$-form $\omega$ on $M$, we denote by $\omega_0$ its restriction to the projection of the flat distribution $T_0$ on $M$. 
In particular, if $\theta$ is the Lee form of $D$ with respect to $g$, we denote by $\theta_0$ the 1-form determined in this way. The functional we consider is then defined as follows:
$$\mathcal{F}\colon \mathcal{M}_1\to \mR, \quad  \mathcal{F}(g):=\int_M|\theta_0|^2_g\mathrm{vol}_g.$$

\begin{prop}
	A metric $g\in  \mathcal{M}_1$ is adapted if and only if $g$ is a critical point of the functional $\mathcal{F}$.
\end{prop}	

\begin{proof}
	For an adapted metric $g$, the $1$-form $\theta_0$ vanishes, so $g$ is a minimum point of $\mathcal{F}$.\\ Conversely, let us first deduce the equation satisfied by a critical point $g$ of $\mathcal{F}$. For this, we consider a variation of $g$ in  $\mathcal{M}_1$, for $|t|$ small enough: 
	$$g_t:=f_t^2g,\qquad f_0\equiv 1.$$
	The fact that $g_t\in \mathcal{M}_1$ implies that $\displaystyle\int_M \dot f_0\,\vol_g=0$.
	The Lee form of $g_t$ is given by $\theta^{t}:=\theta_{g_t}=\theta_g-\di \ln f_t$, and thus its component along the flat part is $\theta_0^t=\theta_0-(\di \ln f_t)_0$.
	We then compute: 
	\begin{equation*}
		\begin{split}
			\mathcal{F}(g_t)&=\int_M|\theta^t_0|^2_{g_t}\,\vol_{g_t}=\int_M f_t^{n-2}|\theta^t_0|^2_{g}\,\vol_g=\int_M f_t^{n-2}|\theta_0-(\di \ln f_t)_0|^2_{g}\,\vol_{g}\\
			&=\int_M (1+(n-2)t\dot f_0+o(t))|\theta_0-t(\di\dot f_0)_0 +o(t) |^2_{g}\,\vol_{g}\\
			&=\int_M (1+(n-2)t\dot f_0+o(t))(|\theta_0|_g^2-2tg(\theta_0,(\di \dot f_0)_0)+o(t))\,\vol_{g}\\
			&=\int_M |\theta_0|_g^2\,\vol_{g}+t\int_M \left[(n-2)\dot f_0 |\theta_0|_g^2 -2g(\theta_0,\di \dot f_0)\right]\vol_{g}+o(t)\\
			&=\int_M |\theta_0|_g^2\,\vol_{g}+t\int_M \dot f_0\left[(n-2)|\theta_0|_g^2 -2\delta^g\theta_0\right]\vol_{g}+o(t)
		\end{split}	
	\end{equation*}	
	Hence $g$ is a critical point of $\mathcal{F}$ if $\frac{\di}{\di t}\big|_{t=0} \mathcal{F}(g_t)=0$ for any variation $g_t$ as above, or equivalently, if  $\displaystyle\int_M f\left[(n-2)|\theta_0|_g^2 -2\delta^g\theta_0\right]\vol_{g}=0$, for all $f\in\mathcal{C}^\infty(M)$ with $\displaystyle\int_M f\,\vol_g=0$. Thus, the equation that a critical point $g$  of $\mathcal{F}$ fulfills is the following:
	$$(n-2)|\theta_0|_g^2 -2\delta^g\theta_0=k, \text { for some } k\in\mR.$$
	Taking the pull-back to $\widetilde M$ and using \eqref{codiffconf}, this equation may be written as follows:
	$$2e^{2\varphi}\delta^h\theta_0+(n-2)|\theta_0|^2_g +k=0,$$
	or, equivalently:
$$2\delta^h\theta_0+(n-2)|\theta_0|^2_h +ke^{-2\varphi}=0.$$
Let $x$ be any point in $N$. Since $\theta_0=(\di\varphi)_0$, if  we denote by $\varphi_0^x:=\varphi|_{\mR^q\times\{x\}}$, then we obtain the following equivalent equation on $\mR^q$:
\begin{equation}\label{eqcrit}
	2\Delta\varphi_0^x+(n-2)|\di\varphi_0^x|^2+ke^{-2\varphi_0^x}=0.
\end{equation}	
On the other hand, the function $\varphi$ is equivariant on $\widetilde M$, so the function $\varphi_0^x$ is bounded on $\mR^q$, according to \cite[Lemma 3.4]{brice}. Hence,  applying \eqref{eqcrit} at a minimum point of $\varphi_0^x$, yields that $k\geq 0$. Thus, it follows from \eqref{eqcrit} that $\Delta\varphi_0^x\leq 0$ on $\mR^q$, \emph{i.e.} $\varphi_0^x$ is a subharmonic function. According to the maximum principle for subharmonic functions, since $\varphi_0^x$ is bounded on $\mR^q$, it must be constant, so $\Delta\varphi_0^x= 0$ and then \eqref{eqcrit} further implies that
$$0\leq (n-2)|\di\varphi_0^x|^2=-ke^{-2\varphi_0^x}\leq 0.$$
Hence $|\di\varphi_0^x|=0$,  so $\varphi_0^x$ is constant on $\mR^q$, which shows that $\varphi$ is a function on $N$, \emph{i.e.} $g$ is an adapted metric.
\end{proof}

\end{document}